\documentclass[a4paper, 11pt]{amsart}
\usepackage[utf8]{inputenc}
\usepackage[T1]{fontenc}
\usepackage{lmodern}
\usepackage[english]{babel}
\usepackage[dvipsnames]{xcolor}
\usepackage{adjustbox}
\usepackage{letltxmacro}
\usepackage{todonotes}
\LetLtxMacro\todonotestodo\todo
\renewcommand{\todo}[2][]{\todonotestodo[#1]{TODO: {#2}}}
\usepackage{enumerate}
\usepackage{hyperref}
\usepackage{url}
\usepackage{mathrsfs}
\usepackage{amssymb}
\usepackage{ stmaryrd }
\theoremstyle{definition}
\newtheorem{lemma}{Lemma}[section]

\theoremstyle{definition}
\newtheorem{theorem}[lemma]{Theorem}

\newtheorem{remark}[lemma]{Remark}

\newtheorem*{claim*}{Claim}

\newtheorem{question}[lemma]{Question}
\newtheorem*{theorem*}{Theorem}
\newtheorem*{corollary*}{Corollary}
\newtheorem*{lemma*}{Lemma}
\newtheorem*{remark*}{Remark}
\newtheorem*{question*}{Question}

\newcommand{\Q}{\mathbb{Q}}
\newcommand{\Z}{\mathbb{Z}}
\newcommand{\N}{\mathbb{N}}
\newcommand{\F}{\mathbb{F}}

\DeclareMathOperator{\Int}{Int}

\title[Irreducibility properties of Carlitz' binomial coefficients]{Irreducibility properties of Carlitz' binomial coefficients for algebraic function fields}

\author{Robert Tichy}
\author{Daniel Windisch}

\thanks{D.~Windisch is supported by the Austrian Science Fund (FWF): I~4406}

\begin{document}

\begin{abstract}
We study the class of univariate polynomials $\beta_k(X)$, introduced by Carlitz, with coefficients in the algebraic function field $\F_q(t)$ over the finite field $\F_q$ with $q$ elements. It is implicit in the work of Carlitz that these polynomials form a $\F_q[t]$-module basis of the ring $\Int(\F_q[t]) = \{f \in \F_q(t)[X] \mid f(\F_q[t]) \subseteq \F_q[t]\}$ of integer-valued polynomials on the polynomial ring $\F_q[t]$. This stands in close analogy to the famous fact that a $\Z$-module basis of the ring $\Int(\Z)$ is given by the binomial polynomials $\binom{X}{k}$.

We prove, for $k = q^s$, where $s$ is a non-negative integer, that $\beta_k$ is irreducible in $\Int(\F_q[t])$ and that it is even absolutely irreducible, that is, all of its powers $\beta_k^m$ with $m>0$ factor uniquely as products of irreducible elements of this ring. As we show, this result is optimal in the sense that $\beta_k$ is not even irreducible if $k$ is not a power of $q$.
\end{abstract}

\maketitle

\section{Introduction}

The well-known binomial polynomials 
\[
\binom{X}{k} = \frac{X(X-1)\cdots(X-k+1)}{k!}.
\]
have the remarkable property that they (interpreted as functions) map integers to integers. In other words, they are elements of the ring
\[\Int(\Z) = \{f \in \Q[X] \mid f(\Z) \subseteq \Z\}\]
of integer-valued polynomials (sometimes called \textit{numerical polynomials} in the context of Hilbert functions of algebraic varieties).

It is documented that already Isaac Newton used these polynomials for interpolation of functions $\Z \to \Z$; see for instance~\cite[page xiii]{Cahen-Chabert}. In the core of this interpolation method lies the fact that the binomial polynomials form a $\Z$-module basis of $\Int(\Z)$.

In 1940, Carlitz~\cite{Carlitz} introduced an analogous class of polynomials in the case of algebraic function fields. For a prime power $q = p^n$ and two independent variables $t$ and $X$ over the finite field $\F_q$, he defined
\[ \psi_0(X) = X, \hspace{0.3cm} \psi_m(X) = \prod_{\substack{{f \in \F_q[t]} \\ {\deg(f) < m}}} (X - f) \]
and
\[ F_0 = 1, \hspace{0.3cm} F_m = (t^{q^m} - t) (t^{q^{m-1}} - t)^q \cdots (t^q - t)^{q^{m-1}}, \]
for $m\geq 1$, where $\deg(0) = - \infty$ as usual. For a non-negative integer $k$ given in $q$-ary digit representation
\[ k = \alpha_0 + \alpha_1 q + \ldots + \alpha_s q^s, \]
for some $s$ and digits $\alpha_i \in \{0, \ldots, q-1\}$, Carlitz introduced
\[ G_k(X) = \psi_0^{\alpha_0}(X) \cdots \psi_s^{\alpha_s}(X) \text{ and } \]
\[ g_k(X) = F_0^{\alpha_0} \cdots F_s^{\alpha_s}. \]
Based on this, we write
\[ \beta_k(X) = \frac{G_k(X)}{g_k} \]
and call this polynomial in the variable $X$ with coefficients in the algebraic function field $\F_q(t)$ the $k$-th \textit{Carlitz binomial polynomial}.

Let us shortly illustrate the analogy of the Carlitz binomial polynomials to the classical binomial polynomials. 
As noted above, the binomial polynomials $b_k = \binom{X}{k} = \frac{X(X-1)\cdots (X-k+1)}{k!}$ form a regular basis of $\Int(\Z)$, that is, they are a $\Z$-module basis of $\Int(\Z)$ and $\deg(b_k) = k$ for each $k \in \N_0$. It is implicit in the work of Carlitz~\cite{Carlitz} that the analogous statement is true for~$\beta_k$.

\begin{remark} \phantom{} 
\begin{enumerate}
\item  $\beta_k \in \Int(\F_q[t]) = \{f \in \F_q(t)[X] \mid f(\F_q[t]) \subseteq \F_q[t]\}$.
\item  $\beta_k$ has degree $k$.
\item $(\beta_k)_{k \in \N_0}$ is a regular basis of $\Int(\F_q[t])$.
\end{enumerate}
\end{remark}

\begin{proof} \phantom{}
\begin{enumerate}
\item This is~\cite[Lemma 2]{Carlitz}.
\item This follows from the paragraph after (1.8) in~\cite{Carlitz}.
\item Each polynomial $f \in \F_q(t)[X]$ has a unique representation $f(X) = \sum_{k = 0}^n A_k G_k(X)$ with $A_k \in \F_q(t)$~\cite[(3.1)]{Carlitz}. Such an $f$ is in $\Int(\F_q[t])$ if and only if $A_kg_k \in \F_q[t]$~\cite[Theorem 9]{Carlitz}. So in this case, by setting $B_k = A_kg_k$ we infer a unique representation $f = \sum_{k = 0}^n B_k \beta_k$, where $B_k \in \F_q[t]$.
\end{enumerate}

\end{proof}

Since the time of Carlitz, the theory of integer-valued polynomials has developed and shown interesting connections to other fields, especially to number theory. For instance, it played a certain role in the work of Barghava on $P$-orderings, see his 2009 paper~\cite{Bhargava} and the survey~\cite{cahen-bhargava}. In the spirit of building bridges between structural results on integer-valued polynomials and other relevant fields, Rissner and the second author~\cite{binomial} employed results from diophantine number theory in order to show that the binomial polynomials, which had been known to be irreducible in $\Int(\Z)$, are in fact \textit{absolutely irreducible}, that is, all of their powers factor uniquely as products of irreducible elements. Apart from that, multiplicative properties of integer-valued polynomials have widely been studied over the last few decades, see for instance~\cite{elasticitycahen,elasticitychapman,FFW,FW,integers,graph,
globalcase,splitDVR}. In particular, we want to mention the work of Sophie Frisch~\cite{integers}, where she showed that for every finite multi-set of integers $\geq 2$ there exists an integer-valued polynomial over $\Z$ whose set of lengths of factorizations into irreducible elements of $\Int(\Z)$ coincides with the given multi-set. In~\cite{globalcase}, this result was extended to a wide class of Dedekind domains, including $\F_q[t]$.

In the following, we mention some work where Carlitz polynomials were applied in different fields. For instance, in~\cite{ergodic} they were used for studying ergodic theory in the power series ring over $\F_2$. In~\cite{1-Lipschitz}, measure preserving $1$-Lipschitz functions on the power series ring of a finite field are characterized by the basis of Carlitz binomial polynomials. Wagner~\cite{Wagner}, a student of Carlitz, explicitly uses the Carlitz binomial polynomials to characterize linear operators. Thakur~\cite{ThakurHypergeometric} developes a whole theory of hypergeometric functions on function fields using properties of Carlitz polynomials. Armana~\cite{Armana}, uses Carlitz binomials coefficients to describe certain values of Zeta functions;
Perkins~\cite{Perkins} studies $L$-series and uses an interpolation formula involving $\beta_k$.
Wagner~\cite{WagnerFact} already studied factorization properties of the Carlitz polynomials. More precisely, he investigated the irreducible factors of certain values of these polynomials.

The purpose of the present paper is to characterize absolute irreducibility of Carlitz binomial polynomials. 

\begin{theorem}\label{theorem:main}
Let $q$ be a prime power and $k$ a positive integer.
\begin{enumerate}

\item The Carlitz binomial polynomial $\beta_k$ is irreducible in $\Int(\F_q[t])$ if and only if $k = q^s$. More precisely, there is a decomposition $\beta_k = \prod_{i = 1}^s (\beta_{q^i})^{\alpha_i}$, where $k = \alpha_0 + \alpha_1 q + \ldots + \alpha_s q^s$ with $\alpha_i \in \{0,\ldots,q-1\}$, $\alpha_s \neq 0$ and $s$ a non-negative integer.

\item The $m$-th power $\beta_{q^s}^m$ of the irreducible Carlitz binomial polynomial $\beta_{q^s}$ factors uniquely as a product of irreducible elements of $\Int(\F_q[t])$ for every positive integer $m$. In other words, $\beta_{q^s} \in \Int(\F_q[t])$ is absolutely irreducible.
\end{enumerate}
\end{theorem}

\section{Number theoretic and combinatorial preparations for the theorem}

\begin{lemma}\label{lemma:inequality}
Let $n = q^s$ where $q\geq 2$ and $s\geq 0$ are integers. Moreover, let $k \in \{0,\ldots,n\}$ and write 
\begin{align*}
k &= \gamma_0 + \gamma_1 q + \ldots + \gamma_s q^s, \\
n-k &= \delta_0 + \delta_1 q + \ldots + \delta_s q^s
\end{align*}
with $\gamma_i, \delta_i \in \{0,\ldots,q-1\}$.
Then
\[
\sum_{i = 0}^s (\gamma_i + \delta_i)  iq^i \leq sq^s
\]
and this inequality is strict unless $k = 0$ or $k = n$.
\end{lemma}

\begin{proof}
Clearly, in the cases $k = 0$ and $k = n$ equality holds, so the statement is trivial. For the remainder of the proof, we assume $0<k<n$. Thus, $\gamma_s = \delta_s = 0$ and we show the strict inequality by induction on $s$. The case $s = 1$ is trivial because the left side of the inequality is $0$ while the right side is $q>0$. So let $s>1$. We distinguish three cases.

\textit{Case 1}. $\gamma_{s-1} + \delta_{s-1} \geq q$. Since
\[
q^s = n = k + (n-k) = \sum_{i = 0}^{s-1} (\gamma_i + \delta_i)  q^i > q^s
\]
if $\gamma_{s-1} + \delta_{s-1} > q$, we infer that $\gamma_{s-1} + \delta_{s-1} = q$ and $\gamma_i + \delta_i = 0 $ for $i < s-1$. From this, it follows that
\[
\sum_{i = 0}^s (\gamma_i + \delta_i)  iq^i = q(s-1)q^{s-1} = (s-1)q^s < sq^s.
\]

\textit{Case 2}. $\gamma_{s-1} + \delta_{s-1} = q-m \leq q-2$, that is, $m\geq 2$. Using $\gamma_i, \delta_i \leq q-1$, we deduce as above that
\begin{align*}
 q^s = n = \sum_{i = 0}^{s-1} (\gamma_i+ \delta_i)q^i &= \sum_{i = 0}^{s-2} (\gamma_i+ \delta_i)q^i + (q-m)q^{s-1}.
\end{align*}
Furthermore, we have
\begin{align*}
 \sum_{i = 0}^{s-2} (\gamma_i+ \delta_i)q^i &\leq 2(q-1) \sum_{i =0}^{s-2} q^i = 2(q^{s-1} -1),
\end{align*}
and thus, we obtain
\[
q^s \leq 2q^{s-1} - 2 + q^s - mq^{s-1} = (2-m)q^{s-1}-2 + q^s < q^s.
\]
This is a contradiction, so the assumption of \textit{Case 2} is impossible.

\textit{Case 3}. $\gamma_{s-1} + \delta_{s-1} = q-1$. As before, we see that
\[
q^s = n = \sum_{i = 0}^{s-2} (\gamma_i+ \delta_i)q^i + (q-1)q^{s-1}
\]
and hence $\sum_{i = 0}^{s-2} (\gamma_i+ \delta_i)q^i = q^{s-1}$. This enables us to apply the induction hypothesis to $n' = q^{s-1}$, $k' = \sum_{i = 0}^{s-2} \gamma_i q^i$ and $n' - k' = \sum_{i = 0}^{s-2} \delta_i q^i$. This leads to
\begin{align*}
\sum_{i = 0}^s (\gamma_i + \delta_i)  iq^i  &= \sum_{i = 0}^{s-2} (\gamma_i + \delta_i)  iq^i + (q-1)(s-1)q^{s-1} \\
	&< (s-1)q^{s-1} + (s-1)(q-1)q^{s-1} = (s-1)q^s < sq^s
\end{align*}
which completes the proof.

\end{proof}

Following the analogy of $\beta_k$ with the binomial polynomials $\binom{X}{k}$, we may consider $g_k$ as an $\F_q[t]$-analogue of $k!$. This leads us to the definition of a binomial coefficient in $\F_q[t]$ as $\binom{n}{k}_{\F_q[t]} = \frac{g_n}{g_k g_{n-k}}$.

\begin{remark}\label{remark:binom}
Let $n$ and $k$ be non-negative integers and write
\begin{align*}
n &= \alpha_0 + \alpha_1 q + \ldots + \alpha_s q^s, \\
k &= \gamma_0 + \gamma_1 q + \ldots + \gamma_s q^s, \\
n - k &= \delta_0 + \delta_1 q + \ldots + \delta_s q^s,
\end{align*}
where $\alpha_i,\gamma_i,\delta_i \in \{0,\ldots,q-1\}$ and $\alpha_s \neq 0$.

Then the following are equivalent:
\begin{enumerate}
\item[(a)] $\binom{n}{k}_{F_q[t]} = 1$
\item[(b)] $\binom{n}{k}_{F_q[t]} \in \F_q^\times$

\end{enumerate}
Moreover, the following statement (c) is sufficient for (a) and (b), and all three are equivalent in case $n = q^s$:
\begin{itemize}
\item[(c)] $\alpha_i = \gamma_i + \delta_i$ for all $i \in \{0,\ldots,s\}$
\end{itemize}
\end{remark}

\begin{proof}
The equivalence of (a) and (b) is obvious since all factors of the $g_k$ are monic. The implication from (c) to (a) is clear. For the converse in case $n = q^s$, note that by Lemma~\ref{lemma:inequality} $\deg_t(g_n) = \sum_{i = 0}^s \alpha_i \cdot i q^i = sq^s \geq \sum_{i = 0}^s \gamma_i \cdot i q^i + \sum_{i = 0}^s \delta_i \cdot i q^i = \deg_t(g_k \cdot g_{n-k})$, with a strict inequality if (c) does not hold.
\end{proof}

We recursively define a family of integer matrices that is needed in the proof of our main result. Let $q$ and $k$ be positive integers. For the moment, $q$ does not have to be a prime power. Let $M_1^{(k)}$ be the $(q\times q)$-matrix that has the entry $k$ along its diagonal and $k-1$ as off-diagonal entries. Thus, $M_1^{(k)}$ is the matrix
\[
\begin{bmatrix} 
      k & k-1 & k-1 &\dots & k-1 \\
    k-1 & k & k-1 & \dots & k-1  \\
    k-1 & k-1 & k &  & \\
    \vdots & \vdots &  & \ddots &  \\
    k-1 & k-1 & &  & k  \\
\end{bmatrix}.
\]
Now, for $i \in \{2,\ldots,k\}$, we define the $(q^i \times q^i)$-matrix $M_i^{(k)}$ recursively as follows: It has $q$ times the block $M_{i-1}^{(k)}$ along its diagonal and the entries outside this block diagonal are $k-i$. For notational simplicity, we set $M_k = M_k^{(k)}$. Note that the entries of $M_k$ are exactly the integers $0,1,\ldots,k$.

\begin{lemma}\label{lemma:matrix}
The matrix $M_k$ has non-zero determinant.
\end{lemma}

\begin{proof}
For $k=1$, $M_k$ is just the unit matrix, which has non-zero determinant. So let $k>1$.
Of course, in order to show that $M_k$ has non-zero determinant it is sufficient to prove that, after row transformations that keep the order of rows stable, have the diagonal elements as Pivot elements, and lead to an upper triangular matrix, all the diagonal elements are positive. Since all the off-diagonal entries of $M_k$ are $\leq k-1$, it is suffient to show this for the matrix that has the entry $k$ on the diagonal and $k-1$ elsewhere (instead of showing it for $M_k$ itself). The $n$-th diagonal entry of this matrix after this elimination process is $k- r_n$, where $r_n$ is given by the following recursive sequence:
\begin{align*}
r_1 &= 0  \text{ and} \\
r_{n+1} &= r_n + ((k-1)-r_n)\cdot \frac{k-1}{k}\\
		&= r_n \cdot (1 - \frac{k-1}{k}) + \frac{(k-1)^2}{k}
\end{align*}
for $n > 1$. We show by induction on $n$ that $r_n < k$ for all $n$, which completes the proof. For $n = 1$ this is certainly true. So let $n\geq1$ and suppose that $r_n <k$. Then
\begin{align*}
r_{n+1} &= r_n \cdot (1 - \frac{k-1}{k}) + \frac{(k-1)^2}{k} \\
		&< k \cdot \frac{1}{k} + \frac{k^2-2k+1}{k}\\
		&= \frac{k^2 - k +1}{k} = k - 1 + \frac{1}{k},
\end{align*}
which is $<k$ because $k>1$.
\end{proof}

\section{Proof of the main result}

We are now ready to prove the asserted irreducibility properties of the Carlitz binomial polynomials.

\begin{proof}[Proof of Theorem~\ref{theorem:main}]
To see (1), first note that by definition of $\beta_k$, this polynomial can be decomposed as $\beta_k = \prod_{i = 1}^s (\beta_{q^i})^{\alpha_i}$. A fortiori, it is reducible in $\Int(\F_q[t])$ in case $k$ is not a power of $q$.

For the converse implication, let $k = q^s$ be a power of $q$ and $G,H \in \Int(\F_q[t])$ such that $\beta_k = G\cdot H$. We show that either $G$ or $H$ must be a unit of $\Int(\F_q[t])$, that is, a non-zero element of $\F_q$. Let $c = \deg(G)$, $d = \deg(H)$ and note that $c + d = k$. 

Since the $\beta_i$ form a regular basis of $\Int(\F_q[t])$ we know that the leading coefficients of $G$ and $H$ are of the form $\frac{g}{g_c}$ and $\frac{h}{g_d}$ for some $g,h \in \F_q[t]$. Hence $\frac{1}{g_k} = \frac{gh}{g_c g_d}$ and therefore $\frac{g_c g_d}{g_k} \in \F_q[t]$. Since in the decomposition $k = \alpha_0 + \alpha_1 q + \ldots + \alpha_s q^s$ the coefficients are $\alpha_0 = \ldots = \alpha_{s-1} = 0$ and $\alpha_s = 1$, we infer by Remark~\ref{remark:binom} that $\frac{g_c g_d}{g_k} \notin \F_q[t]$ unless $c = k$ or $d = k$.

Suppose, without loss of generality, that $d = k$. It follows that $c = 0$ and therefore $G \in \F_q[t]$. As a consequence $\frac{\beta_k}{G} = H \in \Int(\F_q[t])$. Using the unique representation $\frac{\beta_k}{G} = \sum_{i = 0}^k B_i \beta_i$ with $B_i \in \F_q[t]$, we get that $\frac{1}{G} = B_k \in \F_q[t]$ and hence $G \in \F_q[t]^\times$. It follows that $\beta_k$ is irreducible.\\

Now we head towards the proof of (2).
For $s = 0$, $\beta_{q^s} = \beta_1 = X$, and so the statement is trivial. Suppose now that $s > 0$. We use the matrix approach developed by Rissner and the second author~\cite{binomial} in order to show that the binomial polynomials $\binom{x}{n}$ are absolutely irreducible in $\Int(\Z)$. We recall the idea of this approach and, in the course of that, directly apply it to our situation:

By definition, 
$
\beta_{q^s}(X) = \frac{\psi_s(X)}{F_s},
$
and Carlitz proved that $\beta_{q^s}(t^s) = 1$~\cite[Paragraph after (1.4)]{Carlitz}. Using this, we see that 
\[
\beta_{q^s}(X) = \prod_{j = 1}^{q^s} \frac{X - f_j}{t^s - f_j},
\]
where $f_1,\ldots,f_{q^s}$ are all the polynomials in $\F_q[t]$ of degree less than $s$. We do the indexing of the $f_i$ in a way such that $f_1 = 0$ and polynomials of the same non-zero lowest degree term appear consecutively, beginning with degree $s-1$ and ending with degree $0$. To illustrate this, we write down such an indexing for a minimal example. If $q = 3$ and $s = 2$ one possible ordering is
\[
0, t, 2t, 1, 2, t+1, t+2, 2t+1, 2t+2.
\]
We start with $0$, which is followed by the two polynomials with lowest degree non-zero term of degree $1$, and we end with the six polynomials of lowest degree non-zero term of degree $0$.

According to the methods of the mentioned paper~\cite{binomial}, we now build an integer matrix $\mathbf{A}_{q^s}$ whose entries are the values (with respect to a certain valuation on the quotient field $\F_q(t)$ of $\F_q[t]$) of evaluations at $X$ of the individual factors $\frac{X-f_j}{t^s - f_j}$ of $\beta_{q^s}(X)$. While the authors of~\cite{binomial} had to use many distinct valuations on $\Q$, we only need the $t$-adic valuation $\mathsf{v} = \mathsf{v}_t$, thanks to the presence of the transcendental element $t \in \F_q(t)$.

We evaluate the factors $\frac{X-f_j}{t^s - f_j}$ at the polynomial $t^s + f_i$ for $i \in \{2,\ldots,q^s\}$, apply $\mathsf{v}$ and get
\[
\mathsf{v}\left(\frac{(t^s + f_i)-f_j}{t^s - f_j}\right) = \mathsf{v}((t^s + f_i) - f_j) - \mathsf{v}(t^s - f_j).
\]
The result is an integer matrix $\mathbf{A}_{q^s}$ of size $(q^s-1) \times q^s$ whose entry in row $i \in \{2,\ldots,q^s\}$ and column $j \in \{1,\ldots,q^s\}$ is $\mathsf{v}((t^s + f_i) - f_j) - \mathsf{v}(t^s - f_j)$. Moreover, all the rows of $\mathbf{A}_{q^s}$ sum up to $0$ which, finally, makes the matrix approach to absolute irreducibility~\cite{binomial} applicable.

So, following the ideas of~\cite[Proposition 3.15]{binomial}, we only need to show that the rank of $\mathbf{A}_{q^s}$ is $q^s - 1$. In order to do this we prove that the square submatrix $\mathbf{B}_{q^s}$ consisting of columns with indices $j \in \{2,\ldots,q^s\}$ has full rank, that is, it is a regular matrix over $\Q$.

 The way we ordered the $f_i$ makes $\mathbf{B}_{q^s}$ a lower block diagonal matrix. Indeed, if $i$ and $j$ are chosen such that the lowest degree non-zero term of $f_i$ has strictly larger degree then the one of $f_j$ then $\mathsf{v}((t^s+f_i)-f_j) = \mathsf{v}(t^s-f_j)$ and hence the entry of $\mathbf{B}_{q^s}$ at position $(i,j)$ is zero by definition. 

The diagonal consists of $s$ square blocks. The $k$-th block has dimension $(q-1)q^{k-1}$ and contains only positive entries whose values are $0,\ldots,k$. Each row and each column of the $k$-th block contains the entry $k$ exactly once and the entry $k-r$, for $r \in \{1,\ldots,k\}$, exactly $(q-1)q^{r-1}$ times. Moreover, the sets of column indices of the entries $\geq a$ (for $a \in \{0,\ldots,k\}$) of two arbitrary rows of the $k$-th block are either equal or disjoint, and they have the same cardinality. Identifying those rows where these sets are equal, these sets form a partition of the set of column indices of the $k$-th block. The dual statement is true for the row indices of the $k$-th block.

Hence, after row and column permutation inside the $k$-th block (which does not change the lower block diagonal form of the matrix), we arrive at the following symmetric structure for the $k$-th block of the matrix: The matrix $M_1^{(k)}$ of the first $q$ rows and columns has the entry $k$ on its diagonal and $k-1$ elsewhere. The matrix $M_2^{(k)}$ of the first $q^2$ rows and columns has $q$ times the matrix $M_1^{(k)}$ along its diagonal and the entry $k-2$ elsewhere. Accordingly, the matrix $M_i^{(k)}$, for $i \in \{2,\ldots,k\}$, of the first $q^i$ rows and columns of the $k$-th block has $q$ times the matrix $M_{i-1}^{(k)}$ along its diagonal and the entry $k-i$ elsewhere. Note that $M_k^{(k)}$ is exactly the matrix $M_k$ of Lemma~\ref{lemma:matrix} and that it agrees with the $k$-th block. Lemma~\ref{lemma:matrix} says that $M_k$ is regular which finishes the proof of (2).
\end{proof}

%

Independently of the irreducibility of an element of $\Int(\F_q[t])$, one could ask whether the uniqueness property of Theorem~\ref{theorem:main}(2) holds. This leads us to the following open question.

\begin{question}\label{question:general}
Does every power $\beta_{k}^m$ of the Carlitz binomial polynomial $\beta_{k}$ factor uniquely as a product of irreducible elements of $\Int(\F_q[t])$, in particular, when $k$ is not a power of $q$?
\end{question}

\bibliographystyle{amsplainurl}
\bibliography{bibliography}
 
\bigskip

\noindent
\textsc{Robert Tichy, Department of Analysis and Number Theory, Technische Universität Graz,  Steyrergasse 30/II, 8010 Graz, Austria} \\
\textit{E-mail address}: \texttt{tichy@tugraz.at} \\

\noindent
\textsc{Daniel Windisch, Department of Analysis and Number Theory, Technische Universität Graz, Kopernikusgasse 24/II, 8010 Graz, Austria} \\
\textit{E-mail address}: \texttt{dwindisch@math.tugraz.at}

%

\end{document}